\documentclass[11pt]{amsart}
\usepackage{amscd,amssymb,palatino,amsmath}
\usepackage[all,cmtip]{xy}
\usepackage{pb-diagram,pb-xy} % commuting diagrams
\usepackage[utf8]{inputenc}
\usepackage{wrapfig}
\usepackage{xcolor}
\usepackage{color}
\usepackage{float}
\usepackage{amscd,amssymb,palatino,amsmath}
\usepackage{amsthm}
\usepackage{amsfonts}
\usepackage{graphicx}
\usepackage{appendix}
\usepackage{graphicx, xcolor}
\usepackage{amsmath}
\usepackage{mathrsfs}
\usepackage{caption}
\usepackage{enumerate}

\newcommand{\Z}{\mathbb{Z}}

\newcommand{\N}{\mathbb{N}}
\newcommand{\R}{\mathbb{R}}

\renewcommand{\mod}{\textup{mod}}

\newcommand{\frbd}[1]{\frac{\partial}{\partial #1}}

\newtheorem{proposition}{Proposition}
\newtheorem{theorem}[proposition]{Theorem}
\newtheorem{definition}[proposition]{Definition}

\newtheorem{corollary}[proposition]{Corollary}
\newtheorem{remark}[proposition]{Remark}

\newtheorem{example}[proposition]{Example}

\title{A $b$-symplectic slice theorem}
\author{Roisin Braddell}\address{ Roisin Braddell,
Laboratory of Geometry and Dynamical Systems, Department of Mathematics, Universitat Polit\`{e}cnica de Catalunya and BGSMath, Barcelona  \it{e-mail: roisin.braddell@upc.edu }
}

\author{Anna Kiesenhofer}\address{ Anna Kiesenhofer,
Department of Mathematics, EPFL, Lausanne  \it{e-mail:  anna.kiesenhofer@epfl.ch}
 }

\author{Eva Miranda}\address{ Eva Miranda,
Laboratory of Geometry and Dynamical Systems, Departament of Mathematics, Universitat Polit\`{e}cnica de Catalunya BGSMath, Barcelona and Centre de Recerca Matemàtica
 \it{e-mail: eva.miranda@upc.edu}
}

 \thanks{ Eva Miranda  is supported by the Catalan institution for Research and Advanced Studies via an ICREA Academia Prize 2016 and an ICREA Academia Prize 2021 and by the Spanish State
Research Agency, through the Severo Ochoa and Mar\'{\i}a de Maeztu Program for Centers and Units
of Excellence in R\&D (project CEX2020-001084-M) and through the project PID2019-103849GB-I00 / AEI / 10.13039/501100011033. Roisin Braddell is supported by the Severo Ochoa Program SEV-2017-0718, Basque Government BERC Program 2018-2021 and was supported by a predoctoral grant from UPC with the ICREA Academia project of Eva Miranda.   We acknowledge support from the Fondation Sciences Mathématiques de Paris via the Chaire d'Excellence of Eva Miranda supported by a public grant overseen by the French National Research Agency (ANR) as part of the "Investissements d'Avenir" program (reference: ANR-10-LABX-0098) to finance a research stay of Roisin Braddell and a research visit of  Anna Kiesenhofer in Paris to start this project.}

\date{\today}

\begin{document}

\begin{abstract} In this article, motivated by the study of symplectic structures on manifolds with boundary and the systematic study of $b$-symplectic manifolds started in \cite{guillemin2014symplectic}, we prove a slice theorem for Lie group actions on $b$-symplectic manifolds.
\end{abstract}

\maketitle

\section{Introduction}

The linearization of Lie group actions for compact groups in a neighbourhood of a fixed point is due to Bochner \cite{bochner}. It gives a precise description of the local normal form for a Lie group action in the neighbourhood of a fixed point. The next level of difficulty in describing group actions is in the neighbourhood of an invariant submanifold for the action: It was not until the work of Palais in the 60's that such a portrayal was achieved \cite{palais1, palais2}.
The description of the normal form in this case is semilocal (in a neighbourhood of an orbit) and it is done in terms of the action of the group along the orbit and on the \emph{linearization} of its (orthogonal) complement, which Palais denominated "slice".

{ The existence of slices re-conducts the computation of the orbits for the action in terms of its normal space on which it acts linearly. When additional geometrical structures are added into the picture, the existence of normal forms gains interest as it can often be adapted to the new ingredient (the geometric structure). This is the case of symplectic structures where Lie group actions are naturally related to the investigation of Hamiltonian symmetries.}

In particular, symplectic slice theorems, give equivariant normal forms around orbits of symplectic group actions { and become particularly handy when computing the orbits of fundamental vector fields of Hamiltonian actions}. For example the Marle-Guillemin-Sternberg normal form \cite{guillemin1990symplectic}, \cite{marle1985modele} or its generalisations, have been used extensively to study the local structure of symplectic manifolds with symmetries.

The purpose of this article is to extend these results to the singular set-up, more precisely to $b$-symplectic manifolds. {In this singular framework the motivation to find equivariant normal forms comes from the study of symmetric manifolds with boundary}. These singular symplectic structures have been intensely studied since their introduction in  \cite{guillemin2014symplectic}. A study of their geometry in the presence of symmetries was initiated in \cite{guillemin2015toric} (see also \cite{gualtieri2017tropical}) which gave global results on the structure of $b$-symplectic manifolds with toric symmetries and also semilocal models.

In \cite{kiesenhofer2017cotangent} two of the authors in this paper described integrable systems as $b$-cotangent lifts of rotations on a Liouville torus to its cotangent bundle. These can be understood as semilocal models for free actions of tori (associated to integrable systems). Motivated by these models in the integrable case, we plan to give linearized models for general actions of Lie groups in the language of the symplectic slice theorem. Many examples motivating the study of more general symmetries comes from the study of non-commutative integrable systems on ($b$-)symplectic manifolds as considered in \cite{kiesenhofermiranda} and \cite{bolsinov}.

These cotangent lifts appear naturally on the study of geodesic flows:
A $P$-manifold is a Riemannian manifold $M$ with all the geodesics  closed. $2$-dimensional examples of $P$-manifolds  are Zoll and Tannery surfaces (see Chapter 4 in \cite{besse}). In this case the geodesics admit a common period (see Lemma 7.11 in \cite{besse}) inducing an  $S^1$-action on $M$. In the same way that the standard cotangent lift induces a Hamiltonian action on $T^*M$  we can use the twisted $b$-cotangent lift to obtain a $b$-Hamiltonian $S^1$-action on $T^\ast M$. In this case the action is given as a \emph{twisted $b$-cotangent lift} which is a "linear model" of the $b$-Poisson structure parametrized by an additional invariant: a constant $c$. This constant is the  \emph{modular period} of the structure.

The $b$-symplectic slice theorem gives a normal form for a $b$-symplectic form via the symplectic slice theorem, which we recall here (for details on the constructions see \cite{GS} for the Hamiltonian case or \cite{ortega2002symplectic} for the more general symplectic case):

\begin{theorem}\label{symplecticslicetheorem}
Let $(M, \omega)$ be a symplectic manifold and let $H$ be a Lie group acting properly and by symplectomorphisms on $M$. Let $z \in M$. Denote the isotropy group of $z$ by $H_z$ and the orbit of $m$ by $\mathcal{O}_z^H$. Let $V_z$ be the \textit{symplectic normal space}
$$V_{z}:=(T_{z}\mathcal{O}_z^H)^{\omega} /\left((T_{z}\mathcal{O}_z^H)^{\omega} \cap T_{z}\mathcal{O}_z^H \right).$$
%where $H \cdot m$ is the action of the group $H$.
Let $\mathfrak{h}$ be the Lie algebra of $H$ and consider the following subalgebra of $\mathfrak{h}$,
$$\mathfrak{k}:=\left\{
\eta \in \mathfrak{h}\, |\, \eta_M(z) \in (T_{z}\mathcal{O}_z^H)^{\omega}\right\}$$
where $\eta_M$ is the generating vector field of $\eta$.
Let $\mathfrak{i}$ be the Lie algebra of $H_z$ and note that $\mathfrak{i}\subset \mathfrak{k}$. Denote by $\mathfrak{m}$ an $Ad_{H_z}$-invariant complement of $\mathfrak{i}$ in $\mathfrak{k}$.
Then the twisted product
\begin{equation}
Y^H_{z}:=H \times_{H_z}\left(\mathfrak{m}^{*}\times V_{z} \right)
\end{equation}
is a symplectic $H$-space and can be chosen such that there is an $H$-invariant neighbourhood $\mathcal{U}$ of $z$ in $M$, an $H$-invariant neighbourhood $\mathcal{U}'$ of $[e, 0,0]$ in $Y^H_{z}$  and an equivariant symplectomorphism $\phi: \mathcal{U} \rightarrow \mathcal{U}'$ satisfying $\phi(z)=[e, 0,0]$. Equipping the bundle $Y^H_{z}$ with coordinates $[k, \eta, v]$ for $k \in H, \eta \in \mathfrak{m}^{*}$ and $v \in V_z$, $H$ acts on $Y^H_{z}$ as $h\cdot[k, \eta, v]=[h\cdot k, \eta, v]$.
\end{theorem}

\begin{remark}\label{remark:MGS}

{  The symplectic form on the quotient bundle is called the MGS-symplectic form\footnote{MGS for Marle \cite{marle1985modele} and Guillemin-Sternberg \cite{GS}} and it is denoted by $\omega_{MGS}$. This symplectic form constructed as explained above admits an explicit expression in terms of the Lie algebra decomposition associated to the isotropy group (confer \cite{ortega2002symplectic}). A particularly important class of symplectic actions are Hamiltonian actions when the action comes associated with a moment map $\mu: M \to \mathfrak{h}^*$. In the Hamiltonian case, it is possible to describe the moment map for the group action as a splitting.} {   Namely, the moment map $\mu: M \to \mathfrak{h}^*$ may be written as $  \mu([g, \gamma, v]) = Ad^{*}_{g} (\mu(p) + \gamma + \phi(v)),
        $
        where $\phi: V \to \mathfrak{h}_x^*$ is the moment map for the slice representation.  }
        \end{remark}

In order to prove a $b$-symplectic analogue of Theorem \ref{symplecticslicetheorem}, we show that $b$-symplectic manifolds equipped with $b$-symplectic actions transverse to the symplectic foliation possess a finite cover which is a product. The slice theorem then reduces to a ``product slice theorem" modulo the action of a finite group. This linearized model has an additional invariant compared to the symplectic one: the modular period of the component of the critical hypersurface where the orbit lies. We will prove:

\begin{theorem}\label{bslicetheorem}
Let $(M, \omega, G)$ be a $b$-symplectic manifold together with an effective $b$-symplectic action by a compact connected Lie group $G$. Let $Z$ be the critical set of the $b$-symplectic form. Assume that $Z$ is compact and  connected  { and that there is one symplectic leaf of $Z$ which is compact}. Assume further that the orbits of $G$ are transverse to the symplectic foliation of $Z$ . Let $\mathcal{L}$ be a symplectic leaf of $Z$. Then
\begin{enumerate}[(i)]
\item $G$ is necessarily of the form $G=(S^1\times H)/\mathbb{Z}_k$ where
$H$ is a compact, connected Lie group.
\item The action of $G$ lifts to an action of the product group $\tilde{G}=S^1\times H$ on a finite cover $\tilde M$ of a collar neighbourhood of $Z$, ${\tilde M}:=(-\epsilon,\epsilon)\times \tilde{Z}$, $\tilde{Z}\cong S^1\times \mathcal{L}$, where $S^1$ acts on $\tilde{Z}$ by translations on the $S^1$-factor and $H$ by symplectomorphisms on the symplectic leaf $\mathcal{L}$.
\item Let $\tilde{z}\in\tilde{Z}$. Denote by $\mathcal{O}^{\tilde{G}}_{\tilde{z}}$ the orbit of $\tilde{z}$ in $\tilde{Z}$ under the action of $\tilde{G}$ and by $Y^H_{z}$ the bundle of Theorem \ref{symplecticslicetheorem} associated to the action of $H$ on $\mathcal{L}$.  Then there is an equivariant $b$-symplectomorphism from a neighbourhood of the orbit $\mathcal{O}^{\tilde{G}}_{\tilde{z}} \cong S^1\times \mathcal{O}^{H}_{\tilde{z}}$ to the zero section of the bundle $\tilde{E}=T^*S^1 \times Y^{H}_{z}$ where $\tilde{G}/H_{\tilde{z}}$ is embedded as the zero section and the $b$-symplectic form on $\tilde{E}$ is given by
\begin{equation*}%\label{normalform}
\tilde \omega_0=\omega_{c'}+\omega_{MGS}.
\end{equation*}
Here $\omega_{c'}$ is the standard $b$-symplectic form of modular period $c'=kc$ on the manifold $T^*S^1$, see Equation \eqref{twistedsymplectic}, $c$ is the modular period of $Z$ and $\omega_{MGS}$ is the MGS normal form as given by Theorem \ref{symplecticslicetheorem}.
\item Let $\mathcal{O}_z$ be the orbit of $z\in Z$ under the action of $G$. There is an equivariant $b$-symplectomorphism from a neighbourhood of $\mathcal{O}_z$ to a neighbourhood of the zero section of the bundle $E=(T^*S^1 \times Y^H_{z})/\mathbb{Z}_l$ where $\mathbb{Z}_l$ is a finite cyclic group acting by the cotangent lifted action on $T^*S^1$.
\end{enumerate}

\end{theorem}

An important step in the proof is the analysis of the group action along the critical set, $Z$ which is naturally endowed with a cosymplectic structure.
To achieve the proof  we first analyse the consequences of a cosymplectic manifold having a group action transverse to the symplectic foliation. In particular, we prove that given a cosymplectic action of a group $G$ then there are two distinct cases:
\begin{enumerate}
\item $G$ is a group isomorphic to the product of Lie groups $G = S^1 \times H$ or
\item $G=(S^1 \times H) / \Gamma$ where where $\Gamma \subset S^1 \times H$ is of the form $\Gamma=\mathbb{Z}_{l} \times \mathbb{Z}_{k}$ and $\mathbb{Z}_{k}$ is a non-trivial cyclic subgroup of $H$.
\end{enumerate}
We will prove that there is  a finite covering $\tilde{Z}$ of the cosymplectic manifold $Z$ which is trivial (in the sense that $\tilde{Z}\cong S^1 \times \mathcal{L}$) and this finite covering comes equipped with an $S^1\times H$-action which projects to the action of G on Z. Examples of cosymplectic manifolds with cosymplectic symmetries include in particular co-K\"{a}hler manifolds as discussed in \cite{bazzoni2014structure}, which inspired some techniques used here.

We remark that the aim here is to show the rigidity of $b$-symplectic group actions, for which the group action and symplectic form are completely determined in a neighbourhood of an orbit by the isotropy group and its representation on the symplectic normal space. { Therefore, Theorem \ref{bslicetheorem} does not reference the traditional moment map sometimes given as part of the symplectic slice theorem. Notwithstanding,  the combination of our $b$-symplectic slice theorem with the normal form stated in Remark \ref{remark:MGS} yields a normal form for $b$-Hamiltonian actions on $b$-symplectic manifolds.}

\vspace{2mm}

\paragraph{\textbf{Acknowledgements:}} {We are grateful to the anonymous referee for their suggestions and comments that improved this article and increased its readability.}  We are thankful to Konstantinos Efstathiou for  providing us the beautiful picture in this paper. We are thankful to the Fondation Sciences Mathématiques de Paris for financing the stay of Roisin Braddell and Anna Kiesenhofer in Paris and to \emph{Observatoire de Paris} for their hospitality in September 2017-February 2018 during which this project was initiated.

\section{{Preliminaries}}

\subsection{Introduction to $b$-symplectic geometry}

We briefly recall the basics of $b$-symplectic geometry, see \cite{guillemin2014symplectic} for details.

A \textbf{$b$-manifold} is a pair $(M,Z)$ of an oriented manifold $M$ and an oriented hypersurface $Z\subset M$. The hypersurface $Z$ is called the \textit{critical hypersurface}.

A \textbf{$b$-vector field} on a $b$-manifold $(M,Z)$ is a vector field which is tangent to $Z$ at every point $p\in Z$.

If $a$ is a local defining function for the hypersurface $Z$ on some open set $U\subset M$ and $(a,z_2,\ldots,z_{n})$ is a chart on $\mathcal{U}$, then the set of $b$-vector fields on $\mathcal{U}$ is a free $C^\infty(U)$-module with basis
\begin{equation}\label{localframe}
\left(a \frbd{a}, \frbd{z_2},\ldots, \frbd{z_{n}}\right).
\end{equation}
\noindent
The corresponding vector bundle, which exists by the Serre-Swan theorem \cite{swan1962vector}, is the \textit{$b$-tangent bundle}:
\begin{definition}
The \textbf{$b$-tangent bundle}, $^{b} TM$, is the vector bundle whose sections are $b$-vector fields.
\end{definition}

The classical exterior derivative $d$ on the complex of (smooth) $k$-forms extends to the complex of $b$-forms in a natural way. Any $b$-form of degree $k$ can locally be written
\begin{equation*}\label{bkform}
\omega=\alpha\wedge\frac{da}{a}+\beta
\end{equation*}
where $\alpha\in\Omega^{k-1}, \beta\in\Omega^k$, $a$ is a local defining function of $Z$ and $\frac{da}{a}$ is the $b$-one-form dual to $a \frbd{a}$ in a frame of the form \eqref{localframe}. The exterior derivative of $\omega$  is then given by
$$d\omega:=d\alpha\wedge\frac{da}{a}+d\beta.$$
\begin{definition}
A $b$-symplectic form is a $b$-form of degree 2 which is closed and non-degenerate as a $b$-form.
\end{definition}

If $Z$ is the critical hypersurface of a $b$-symplectic form, it can be shown that $Z$ has a codimension-one foliation by symplectic leaves (see \cite{guillemin2011codimension}). The hypersurface $Z$ is then cosymplectic as studied in  \cite{libermann1959automorphismes}.

We recall the following notions for symplectic codimension one foliations given in \cite{guillemin2011codimension}:%, which we will use later on
\begin{definition}
Let $\mathcal{F}$ be a codimension one symplectic foliation of a manifold $Z$. A form $\alpha \in \Omega^1(Z)$ is a \textbf{defining one-form} of $\mathcal{F}$ if it is nowhere vanishing and $\iota_{\mathcal{L}}^*\alpha=0$ for all leaves $\mathcal{L}$, where $\iota_{\mathcal{L}}$ is the inclusion $\mathcal{L}\hookrightarrow Z$, i.e. the kernel of $\alpha$ at any point $p\in Z$ is the tangent space of the leaf through $p$.

A form $\omega \in \Omega^2(Z)$ is a $\textbf{defining two-form}$ of $\mathcal{F}$ if $\iota_\mathcal{L}^*\omega$ is the given symplectic form on each leaf of the foliation.
\end{definition}

If $Z$ is the critical hypersurface of a $b$-symplectic manifold, then the defining one- and two-form of the induced symplectic foliation can be chosen to be \textit{closed} \cite{guillemin2011codimension}. Conversely, a manifold $Z$ with a codimension one symplectic foliation that admits closed defining one- and two-form $\alpha$ resp. $\beta$ can be extended to a $b$-symplectic manifold $M = Z \times \R$ with $b$-symplectic form
$$\omega = \pi_Z^* \alpha \wedge \frac{da}{a} + \pi_Z^* \beta.$$
where $\pi_Z:Z \times \R \to Z$ is the canonical projection and $a$ the coordinate on $\R$.\\

$b$-Symplectic manifolds can also be viewed dually as a particular class of Poisson manifolds. As such they have a modular vector field:
%In this interpretation, $b$-symplectic actions are simply Poisson actions.

\begin{definition}
Let $(M, \Pi)$ be a Poisson manifold equipped with a volume form $\Omega$ and for each $f\in C^{\infty}(M)$ denote by $X_f$ the Hamiltonian vector field associated to $f$. Then the \textbf{modular vector field} of $(M, \Pi)$ is the following derivation on $C^{\infty}(M)$:
$$v_{mod}: C^{\infty}(M) \rightarrow \mathbb{R}: f \mapsto \frac{\mathcal{L}_{X_{f} \Omega}}{\Omega}.$$
\end{definition}
It can be shown that the modular vector field is a Poisson vector field and that the modular vector fields associated to different volume forms only differ by a Hamiltonian vector field.  {The following proposition gives a characterization of the modular vector field for  $b$-symplectic manifolds:}

\begin{proposition}[Proposition 25 in \cite{guillemin2014symplectic}]
{ The modular vector field of a $b$-symplectic manifold $(M, Z)$ is tangent to $Z$ and transverse to the symplectic leaves inside $Z$, independently of the volume form considered on $M$.}
\end{proposition}

Having chosen a modular vector field $v_{mod}$, we can choose defining one and two-forms of the symplectic foliation on $Z$ uniquely by imposing

\begin{equation}\label{emphthe}
    \alpha(v_{mod})=1 \text { and } \iota_{v_{mod}} \omega=0.
\end{equation}

We will call defining one- and two-forms fulfilling this condition \emph{the} defining one- and two-forms of the foliation. They are automatically closed \cite{guillemin2011codimension}.

\begin{remark}
 {Remark that the modular vector field of $Z$ viewed as a Poisson manifold does \emph{not} equal the modular vector field of the $b$-symplectic manifold $M$ restricted to $Z$. In fact, as shown in \cite{guillemin2011codimension}, $Z$ is unimodular when viewed as a Poisson manifold and so the modular vector field vanishes (up to the addition of a Hamiltonian vector field).}
\end{remark}

Finally, we note that the flow of the modular vector field can be used to define the mapping torus structure of $Z$ and define the \emph{modular period} of the $b$-symplectic form as follows (cf. \cite{guillemin2011codimension}):
\begin{definition}\label{mappingtorusstructure}
Let $(M,Z)$ be a $b$-symplectic manifold and suppose that $Z$ is compact and connected and that its symplectic foliation has a compact leaf $\mathcal{L}$. Then $Z$ is a mapping torus. More precisely, taking any modular vector field $v_{mod}$, there exists a  number $c>0$ such that
$$
Z \cong \frac{[0, c]\times \mathcal{L}}{(0,x) \sim(c,\phi(x))}
$$
where the time $t$-flow of $v_{mod}$ corresponds to translation by $t$ in the first coordinate. In particular, $\phi$ is the time $c$-flow of $v_{mod}$. The number $c>0$ is called the \textbf{modular period} of $Z$ and does not depend on the choice of modular vector field $v_{mod}$.
\end{definition}

For notational convenience we will consider rather $t \in [0, 1]$. The modular vector field is then given by
\begin{equation*}
    v_{mod}=\frac{1}{c}\frbd{t}
\end{equation*}
and the defining one-form is given by
$$\alpha = c dt.$$

%It can be shown that if $\omega_0, \omega_1$ are $b$-symplectomorphic in a neighbourhood of $Z$ then they have the same modular periods.

% {\paragraph{Assumptions:} We will assume throughout the article that the critical set $Z$ has  one symplectic leaf of $Z$ which is compact (and thus, in view of the theorem above, all the symplectic leaves are compact and $Z$ is a mapping torus.)}

The
$b$-analogue of the Moser theorem for symplectic manifolds is proved in  \cite{guillemin2014symplectic}.
\begin{theorem}[\textbf{$b$-Moser Theorem}]\label{theorem:bmoser}
Let $\omega_0$ and $\omega_1$ be two $b$-symplectic forms on $(M,Z)$.
If they induce on $Z$ the same corank one Poisson structure and
their modular vector fields differ on $Z$ by a Hamiltonian vector field, then there exist neighbourhoods
$ U_0, U_1$ of $Z$ in $M$ and a diffeomorphism $\gamma: U_0\rightarrow U_1$ such that
$\gamma|_Z=\text{id}_Z$ and $\gamma^*\omega_1=\omega_0$.
\end{theorem}

The condition that  $\omega_0$ and $\omega_1$ induce the same Poisson structure on $Z$ and the same modular vector field (up to a Hamiltonian vector field) is equivalent to demanding that the induced symplectic foliations have the same defining one- and two-forms.

A consequence is the following semilocal model {proved} in \cite{guillemin2014symplectic}:
\begin{corollary}[Extension Theorem]\label{extension}
Let $(M,Z)$ be a $b$-symplectic manifold where $Z$ is compact and connected. Then there is a neighbourhood of $Z$ in $M$ that is $b$-symplectomorphic to the collar neighbourhood $Z\times (-\epsilon,\epsilon)$ with $b$-symplectic form
\begin{equation}\label{semilocal}
\omega=  {\pi_{Z}^*} \alpha \wedge \frac{da}{a} + \pi_{Z}^* \beta.
\end{equation}
where $\alpha$, $\beta$ are the defining one- resp. two-forms on $Z$, $a$  is the coordinate on the interval $(-\epsilon,\epsilon)$ and $\pi_Z$ the projection of the collar to $Z$.
\end{corollary}

As noted in \cite{guillemin2015toric}, by averaging the vector fields of the $b$-Moser theorem, given two $b$-symplectic forms invariant under a group action and $b$-symplectomorphic by the $b$-Moser theorem, we can choose the $b$-symplectomorphism to be equivariant with respect to the group action.  In the special case where $M$ is two-dimensional this yields the following semilocal normal form:
\begin{proposition}
Let $(M,Z)$ be a two-dimensional $b$-symplectic manifold with compact connected critical hypersurface $Z$ and modular period $c>0$. Then $Z \cong S^1$ and there exists a  neighbourhood of $Z$ which is $b$-symplectomorphic to $S^1 \times (-\epsilon,\epsilon)$ with  $b$-symplectic form
\begin{equation}\label{twistedsymplectic}
\omega_c:= c dt \wedge \frac{da}{a}.
\end{equation}
Here $(t,a)$ are the standard coordinates on $S^1 \times \mathbb{R}$.
\end{proposition}

It will be convenient to view the $b$-symplectic manifold  $S^1 \times (-\epsilon,\epsilon)$ as a neighbourhood of the zero section of the cotangent bundle $T^* S^1 \cong S^1 \times \R$ with $b$-symplectic form given by the formula in Equation \eqref{twistedsymplectic}. We also remark for future purposes that $\omega_c$ is clearly invariant under the cotangent lift of the action of $S^1$ on itself by translations.

{In this article we characterize the normal form for actions which preserve a $b$-symplectic form  ($b$-symplectic actions). Among the class of $b$-symplectic actions, the $b$-Hamiltonian class plays a central role.
We end up this section of preliminaries with the definition of $b$-Hamiltonian action. We refer  the interested reader to the articles \cite{guillemin2015toric, GMPS2,GMWconvexity, GMWgeomq}. }

 {
    \begin{definition} \label{def:bHamMM}
        The action of $G$ on a $b$-symplectic manifold $(M, Z, \omega)$ is called $b$-Hamiltonian if there exists a moment map $\mu \in ^{b}\mathcal{C}^\infty(M) \otimes \mathfrak{g}^*$ with
        $$
            \iota(\upsilon_\xi) \omega = \left < d \mu, \xi\right >
        $$
        where $\upsilon_\xi$ is the fundamental vector field associated to $\xi\in \mathfrak{g}$ and the set of $b$-functions is defined as $^{b}\mathcal{C}^\infty (M) = \{a \log |t| + g, g \in \mathcal{C}^\infty (M)\}$.
    \end{definition}}

   { In other words, the action is $b$-Hamiltonian if it preserves the $b$-symplectic form and the form $\iota(\upsilon_\xi) \omega$ is exact.}

\subsection{Transversally equivariant fibrations}

A bundle map $\pi: Z \rightarrow S^1$ is a \textit{transversally equivariant fibration} if there is a smooth $S^1$ -action on $Z$ such that the orbits of the action are transversal to the fibres of $\pi$ and $\pi(t \cdot x)-\pi(x)$ depends on $t \in S^1$ only. The following is a specialization of a theorem by Sadowski  which was applied to the case of co-K\"{a}hler manifolds in \cite{bazzoni2014structure}.

\begin{theorem}\label{sadowski}
Let $Z \stackrel{\pi}{\rightarrow} {S}^{1}$ be a smooth bundle projection from a  closed manifold $Z$ to the circle. The following are equivalent:
\begin{enumerate}
\item $Z \stackrel{\pi}{\rightarrow} {S}^{1}$ is a mapping torus associated to a diffeomorphism of finite order
\item The bundle map $\pi$ is transversally equivariant with respect to an ${S}^{1}$-action on $Z$, $\rho: {S}^{1} \times Z \rightarrow Z$.
\end{enumerate}
Let $\mathcal{L}$ be the fibre of $\pi$. If the above conditions are satisfied then $Z$ has a $\mathbb{Z}_k$-cover ($k \in \N$)
$$p: \tilde{Z}=S^{1} \times \mathcal{L} \rightarrow Z$$
given by the action $(t,l)\mapsto \rho_t(l)$, where $\mathbb{Z}_k$ acts diagonally on $ S^{1} \times \mathcal{L}$ and by translations on $ S^{1}$.
\end{theorem}

Explicitly, we can describe the $\mathbb{Z}_k$-action as follows: Consider the leaf-fixing subgroup of $S^1$,
\begin{equation}\label{discretegroup}
\Gamma = \{s \in S^1 \,: \, \rho_s(\mathcal{L})= \mathcal{L}\}.
\end{equation}
Identifying $S^1\cong \mathbb{R} \,\mod\, 1$, the group $\Gamma$ is of the form $\{0,\frac{1}{k},\ldots,\frac{k-1}{k}\}$ for some $k\in \mathbb{N}$ and hence we can identify it with $\mathbb{Z}_k$ in the natural way. Then for $m\in \mathbb{Z}_k = \{0,1,\ldots, k-1\}$, the action $\rho_{\frac{m}{k}}$ restricts to  a leaf automorphism
\begin{equation}\label{leafaction}
    \sigma_m: \mathcal{L} \rightarrow \mathcal{L},\quad  \sigma_m(l)= \rho_{\frac{m}{k}}(l).
\end{equation}
 The mapping torus $Z$ is then the quotient of the cover $\tilde{Z}$ by the following action of $\mathbb{Z}_k$ on $\tilde{Z}$
\begin{equation}%\label{coverdiscreteaction}
\mu_m(t,l)=(t-\frac{m}{k},\sigma_{m}(l)),\quad m\in \mathbb{Z}_k,\, (t,l)\in \tilde Z = S^1 \times \mathcal{L}.
\end{equation}

From the condition of transverse equivariance, it is clear that $\rho$ maps leaves to leaves. It induces an action on the base $S^1$  given by translations
$t \mapsto t + ks$
and the equivariance condition reads
$$\pi(\rho_s(l)) = k s, \quad l\in  \mathcal{L}:= \pi^{-1}(\{0\}) .$$

%We use the standard parametrization of $S^1$ by $t\in [0,1)$ and the following parametrization for $Z$:
%$$[0,1)\times \mathcal{L} \to Z: (t,l)\mapsto \rho_{\pm \frac{t}{k}}(l).$$
%This is such that the mapping torus projection reads $\pi(t,l)=t$. The cover projection in these coordinates is $p(t,l)=(\pm kt, l)$ for $t\in[0,\frac{1}{k}).$

There is an associated ${S}^1$-action  $\tilde{\rho}$ on the cover $\tilde{Z}$ given by
\begin{equation}\label{covercicleaction}
\tilde{\rho}_s(t,l)=(t+s,l),\quad s\in S^1, (t,l)\in S^1 \times \mathcal{L}.
\end{equation}
The projection $\tilde Z \to Z$ is equivariant with respect to this action.

The existence of a finite trivializing cover of the critical hypersurface $Z$ will play a crucial role in the $b$-symplectic slice theorem.

\section{A trivializing cover for the critical hypersurface}

Now we consider $(M,Z)$ a $b$-symplectic manifold. As we focus on a semi local result, we will assume $M \cong Z\times (-\epsilon,\epsilon)$ where the critical hypersurface $Z$ is compact and connected  with $b$-symplectic form given by Equation \eqref{semilocal}. On a semilocal level the last assumption is not an additional restriction, since as we have seen in the previous section, any $b$-symplectic manifold satisfying the previous conditions is of this form on a tubular neighbourhood of its critical hypersurface. {We will further assume for the rest of the article that $Z$ has a compact leaf. Note that according to \cite{guillemin2011codimension} this implies that $Z$ has a mapping torus structure (though not necessarily the same mapping torus structure given by Theorem \ref{sadowski}).}

\begin{definition}
A group action on a $b$-symplectic manifold is called \emph{transverse} if it is transverse to the symplectic foliation of the critical hypersurface. If the action, restricted to the critical hypersurface, preserves the cosymplectic structure we will call the action \emph{cosymplectic}. Finally, if the action preserves the $b$-symplectic form we will call the action \emph{$b$-symplectic}.
\end{definition}

Cosymplectic and $b$-symplectic actions are special cases of Poisson actions, when considering the manifolds with the associated Poisson structures.

{ As cosymplectic actions are leaf preserving}, cosymplectic actions transverse to the symplectic foliation are automatically transversely equivariant { where the relevant bundle map $\pi: Z \rightarrow S^{1}$ is a projection to the base of the mapping torus. Indeed, the $S^1$-action being leaf-preserving implies by definition that $\pi(t \cdot x)-\pi(x)$ depends only on $t \in S^{1}$.} The next proposition then follows directly from Theorem \ref{sadowski}:

\begin{proposition}\label{finitecover}
Let $Z$ be a cosymplectic manifold and suppose $Z$ has a transverse $S^1$-action preserving the cosymplectic structure. Then $Z$ has a finite cover $\tilde{Z}:=S^1 \times \mathcal{L}$, $\mathcal{L}$ a leaf of the foliation, equipped with an $S^1$-action given by translation in the first coordinate for which the projection $p:S^1 \times \mathcal{L} \rightarrow Z$ is equivariant.
\end{proposition}

To get a cosymplectic structure on the cover, one simply lifts the associated defining one and two-forms.

\begin{proposition}\label{quotientpoissonstructure}
In the setting of the previous proposition, the cosymplectic structure on $Z$ is given by the quotient of a cosymplectic structure on $\tilde{Z}=S^1\times \mathcal{L}$ by the action of a finite cyclic group $\mathbb{Z}_k$.
\end{proposition}

\begin{proof}

Let $p:\tilde{Z} \to Z$ be the finite cover given by Proposition \ref{finitecover}. Denote the one and two forms of the cosymplectic structure by $\alpha$ and $\beta$ respectively. Then $\tilde{\beta}=p^*\beta$ and $\tilde{\alpha}=p^*\alpha$ can easily be shown to define a cosymplectic structure on $S^1\times \mathcal{L}$ and by construction, the cosymplectic structure on the quotient agrees with the cosymplectic structure on $Z$.
\end{proof}

To extend this cover to a $b$-symplectic neighbourhood of $Z$ we simply use the extension theorem (Corollary \ref{extension}):

\begin{corollary}\label{quotientbstructure}
Let $M = Z \times (-\epsilon,\epsilon)$ come equipped with a transverse  $S^1$-action preserving the $b$-symplectic form $\omega$. Then the $b$-symplectic structure on $M$ is $b$-symplectomorphic in a neighbourhood of $Z$ to the quotient of a $b$-symplectic structure on $ S^1\times \mathcal{L} \times (-\epsilon,\epsilon)$ by the action of a finite cyclic group.
\end{corollary}

\begin{proof}
As before let $p:\tilde{Z} \to Z$ be the finite cover. Let $v_{mod}$ be some choice of modular vector field and denote the defining one and two-forms of $Z$ fulfilling the condition in Equation \eqref{emphthe} by $\alpha$ and $\beta$ respectively. Denote by $\tilde{\alpha}, \tilde{\beta}$ the correspond two forms defined in Proposition \ref{quotientpoissonstructure}. By the extension theorem we can assume that the $b$-symplectic form on $M$ is
\begin{equation*}%\label{semilocal}
\omega= \pi_{Z}^* \alpha \wedge \frac{da}{a} + \pi_{Z}^* \beta.
\end{equation*}

Let $\tilde M:= \tilde{Z}\times (-\epsilon, \epsilon)$. Then we have a finite cover $p_M: \tilde M \to M$ for $M$ given by the product map of the cover $p: \tilde Z \to Z$ and the identity on ${(-\epsilon, \epsilon)}$.
Let $\pi_{\tilde{Z}}:\tilde{M} \to \tilde{Z}$ be the projection onto the first factor. Define for $a\in(-\epsilon,\epsilon)$ the $b$-symplectic form on $\tilde M$
$$\tilde{\omega}=\pi_{\tilde{Z}}^*\tilde{\alpha}\wedge\frac{da}{a}+\pi_{\tilde{Z}}^*\tilde{\beta}.$$
Then by construction $(p_M)^* \omega = \tilde \omega$.
\end{proof}

\begin{remark}\label{rem:modularperiod}
Note that the modular period of the associated $b$-symplectic form on the $\mathbb{Z}_k$ cover is  $k$ times the modular period of the $b$-symplectic form on the base. { This can be seen as follows: let $(t,l)\in $ be a coordinate system on $\tilde{Z}=S^1\times \mathcal{L}$. Then the projection $p:\tilde{Z} \rightarrow Z$ acts on the $S^1$ factor as $p(t)=kt \textrm{ mod } 1$. Therefore, if the defining one-form on $Z$ is $\alpha=c dt$ the defining one-form on $\tilde{Z}$ is given by $\tilde{\alpha}=p^{*}(c dt)=c dkt=ck dt$ and the modular period of $\tilde{Z}$ is $ck$.}
\end{remark}

\begin{remark}\label{remark:quotientpoisson}
Similarly, any $b$-symplectic structure with defining one and two-forms $\tilde{\alpha}$ and $\tilde{\beta}$ equipped with a discrete $b$-symplectic group action gives a $b$-symplectic structure on the quotient. For such a group action there are well defined one and two-forms, $\alpha$ and $\beta$, on the base manifold defined by $p^*(\alpha)=\tilde{\alpha}$ and $p^*(\beta)=\tilde{\beta}$, where $p$ is the projection to the quotient. Then $\alpha$ and $\beta$ automatically fulfil the conditions to define a cosymplectic structure on the image of the critical hypersurface. As the group action is discrete, the quotient of the symplectic structure on leaves is likewise symplectic.
\end{remark}

\section{The $b$-symplectic slice theorem for an $S^1$-action}

First, we wish to simplify the expression of the $b$-symplectic form in the neighbourhood of an orbit. In the case that the leaf $\mathcal{L}$ is simply connected, the $b$-symplectic form has a particularly simple expression.

\begin{proposition}\label{theorem:simplyconnected}
Let $M \cong Z\times(-\epsilon,\epsilon)$ be a $b$-symplectic manifold and suppose that $Z$ is a product, $Z\cong S^1 \times \mathcal{L}$, $\mathcal{L}$ a leaf of the symplectic foliation. {Suppose furthermore that $\mathcal{L}$ is simply connected}. Then %there exist coordinates $(t, l, f)$ on $S^1 \times \mathcal{L} \times (-\epsilon,\epsilon)$ in which the $b$-symplectic form is given by
for a suitable defining function $f$ of $Z$ the $b$-symplectic form is given by
\begin{equation}\label{normalform:simplyconnected}
\omega = c dt\wedge\frac{d f}{f}+\pi_{\mathcal{L}}^*(\beta)
\end{equation}
where $t$ is the standard coordinate on $S^1$, $\beta$ is the symplectic form on $\mathcal{L}$ and $\pi_{\mathcal{L}}$ is the projection $ S^1\times \mathcal{L}\rightarrow \mathcal{L}$.
\end{proposition}

\begin{proof}
A $b$-symplectic form on  on $S^1 \times \mathcal{L} \times (-\epsilon,\epsilon)$ equipped with coordinates $(t,l,a)$ can be written
$$\omega=c dt\wedge\frac{da}{a}+dt \wedge \eta+\pi_{\mathcal{L}}^*(\beta)$$
where $\beta$ is the symplectic form on $\mathcal{L}$.

{When $\mathcal{L}$ is simply connected, $\eta=dh$ for some $h\in C^{\infty}(M)$.} The function $f=a e^h$ is then a defining function for $Z$  and moreover
$$\frac{df}{f}=\frac{da}{a}+dh$$
Whence we have
$$\omega=c dt\wedge\frac{df}{f}+\pi_{\mathcal{L}}^*(\beta).$$
\end{proof}

As in the symplectic slice theorem, the normal form of a $b$-symplectic form in the neighbourhood of an orbit is given by virtue of an equivariant Moser theorem. Equivariant $b$-Moser theorems for isotopic forms invariant under $S^{1}$-actions have been given in \cite{guillemin2015toric} and for more general groups in \cite{miranda2018equivariant}. As we wish to compare $b$-symplectic forms in the neighbourhood of an orbit rather than on the whole of $Z$ we require an equivariant $b$-Moser theorem of a slightly different nature:

\begin{proposition}\label{th:equivariantrelativeMoser}
Suppose that $\omega_1$ and $\omega_0$ are $b$-symplectic forms on $M$, invariant under an action of a group $G$ on $M$ which is transverse Poisson for $\omega_1$ and $\omega_0$. Denote by $\mathcal{O}_z$ the orbit of some $z\in Z$ and suppose that $\omega_1$ and $\omega_0$ coincide at $z$. Then  $\omega_1$ and $\omega_0$ are equivariantly $b$-symplectomorphic in some neighbourhood $\mathcal{U}$ of $\mathcal{O}_z$.
\end{proposition}

\begin{proof}
As the defining one and two-forms associated to $\omega_1$ and $\omega_0$ are invariant under the ${S}^1$-action, it follows that on $\mathcal{O}_z$ we have $\alpha_0=\alpha_1$ and $\beta_0=\beta_1$. By the relative Poincar\'{e} lemma, in a contractible neighbourhood of $\mathcal{O}_z$ we have that $\alpha_0-\alpha_1=dg$, an exact one-form on $\mathcal{U}$ and similarly $\beta_0-\beta_1=d\eta$, an exact two-form on $\mathcal{U}$. Whence  $\omega_0-\omega_1=d(-g\frac{df}{f}+\eta)$.
Then $\omega_t=\omega_0+(1-t)\omega_1$ is  non degenerate on $\mathcal{O}_z$ and so on a neighbourhood of $\mathcal{O}_z$. We use  this to define a $b$-vector field $v_t$ by $\iota_{v_t}\omega_t=g\frac{df}{f}-\eta$. As $v_t$ is zero on $\mathcal{O}_z$, the time-one flow exists in a neighbourhood of $\mathcal{O}_z$ and gives the required $b$-symplectomorphism. As both $b$-symplectic forms are invariant under the group action, we can choose the $b$-symplectomorphism to be equivariant.
\end{proof}

\begin{theorem}\label{maintheorem}
Let $M \cong (-\epsilon,\epsilon) \times Z$ be a $b$-symplectic manifold equipped with a $b$-symplectic form $\omega$ of modular period $c$ and a transverse $b$-symplectic $S^1$-action. Let $z\in \mathcal L \subset Z$, let $\mathcal{O}_z$ be its orbit under the $S^1$-action, let $V:=T_z \mathcal L$ and let $\mathbb{Z}_l$ be the isotropy group of $z$. Then there exists an $S^1$-equivariant neighbourhood $(-\epsilon,\epsilon) \times \mathcal{U}$ of $\mathcal{O}_z$ in $M$ and an $S^1$-equivariant mapping
\begin{equation}\label{eq:bsymplectomorphism}
\phi: (-\epsilon,\epsilon) \times \mathcal{U} \rightarrow (T^*S^1\times V)/\mathbb{Z}_l
\end{equation}
where $\mathcal{O}_z$ is embedded as the zero section of the bundle $S^1\times_{\mathbb{Z}_l}\mathbb{R}\times V\cong (T^*S^1\times V)/\mathbb{Z}_l$ and where the action of $\mathbb{Z}_l$ is given by the cotangent lifted action on $T^*S^1$ and by the isotropy representation on $V$.

Moreover, if we equip the bundle $T^*S^1\times V$ with the $b$-symplectic form:
\begin{equation*}
    \tilde{\omega}_0=\omega_{c'}+\omega_V
\end{equation*}
where $\omega_{c'}$ the $b$-symplectic normal form on $T^* S^1$ as given in Definition \ref{twistedsymplectic} with modular period $c'=kc$ and $\omega_V$ the linear symplectic form on $V$, and the quotient $(T^*S^1\times V)/\mathbb{Z}_l$ with the quotient $b$-symplectic form $\omega_0$ (see Remark \ref{remark:quotientpoisson}) then the mapping becomes an equivariant $b$-symplectomorphism onto its image.
\end{theorem}
\begin{figure}
\includegraphics[width = 9cm]{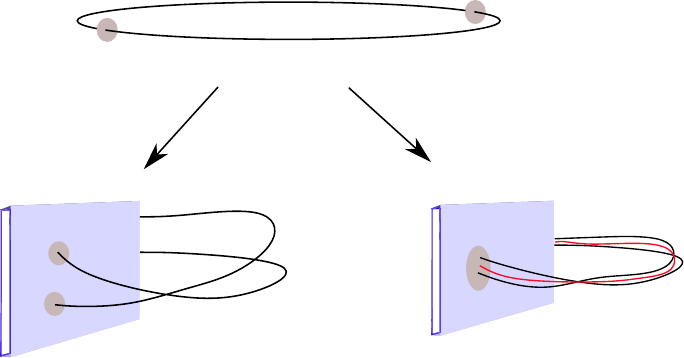}
\caption{A scheme of the trivializing finite cover with a regular orbit ($\Gamma_{{z}}=0$) in black and exceptional orbit ($\Gamma_{{z}}=\mathbb{Z}_2$) in red.}
\end{figure}

\begin{proof}
Let $z\in Z$ be a point in the critical set and $\mathcal{O}_z$ the orbit of $z$ under the $S^1$-action $\rho$. Denote by $\Gamma_{z}$ the isotropy group of $z$. Note that $\Gamma_{z}$ is automatically a subgroup of $\mathbb{Z}_k$ and so $\Gamma_{z} \cong \mathbb{Z}_l$ for some $l$. By the slice theorem there exists a  neighbourhood $\mathcal{U}$ of $\mathcal{O}_z$ in $Z$ equivariantly diffeomorphic to a neighbourhood of the zero section of the vector bundle $S^1 \times_{\Gamma_z} T_{z} Z / T_{z}\mathcal{O}_z$, where $S^1$ acts on the latter according to $s\cdot [t, v]=[t+s, v]$. By choosing the invariant Riemannian metric in the proof of the slice theorem in such a way that $T_{z} \mathcal{L}$ is orthogonal to $T_z \mathcal{O}_z$, the equivariant diffeomorphism can be expressed
$$ S^1 \times_{\Gamma_z} T_{z} \mathcal{L} \to
\mathcal{U}: [t,v] \mapsto \rho_t(\exp_z v).$$
Denote by $\psi$ the corresponding diffeomorphism on the neighbourhood $(-\epsilon,\epsilon) \times \mathcal{U}$ of $\mathcal{O}_z$ in $M$:
$$\psi:(-\epsilon,\epsilon) \times \mathcal{U} \rightarrow (-\epsilon,\epsilon) \times S^1 \times_{\Gamma_z} T_{z} \mathcal{L}.$$
Restricting the defining one and two forms of $\omega$ to $\mathcal{U}$, we have that $\mathcal{U}$ is a cosymplectic manifold with a cosymplectic $S^1$-action. The symplectic leaves of $\mathcal{U}$ are given by $\mathcal{L}_\mathcal{U}:=\mathcal{U} \cap \mathcal{L}$ and the leaf fixing subgroup as defined by Equation \eqref{discretegroup} is $\Gamma_z$.
By Proposition \ref{finitecover} there is a trivial $\Gamma_z$-cover  $\tilde{\mathcal{U}} \cong S^1 \times \mathcal{L}_\mathcal{U}$ of $\mathcal{U}$.
Then $\omega|_{(-\epsilon,\epsilon) \times \mathcal{U}}$
is the quotient of a unique $b$-symplectic form $\tilde{\omega}$ on $(-\epsilon, \epsilon) \times \tilde{\mathcal{U}}$ as given by Corollary \ref{quotientbstructure}. By Proposition \ref{theorem:simplyconnected} we may assume $\tilde{\omega}$ is of the form
$$\tilde{\omega}=c k d t \wedge \frac{d a}{a}+\pi_{\mathcal{L}_\mathcal{U}}^*\beta$$
where $a\in (-\epsilon,\epsilon)$ and $\beta$ is a symplectic two-form given on a leaf $\mathcal{U}_\mathcal{L}$.
Consider the two form $\beta_z$ on $T_{z} \mathcal{L}$. On $(-\epsilon,\epsilon)\times S^1 \times T_{z} \mathcal{L}$ define the $b$-symplectic form
$$\tilde{\omega}_0=c k d t \wedge \frac{d a}{a}+\beta_z.$$
Denote the quotient $b$-symplectic form on $((-\epsilon,\epsilon)\times S^1 \times T_{z} \mathcal{L}))/\Gamma_z$ given in Remark \ref{quotientbstructure} by $\omega_0$. Finally consider the $b$-symplectic form $\psi^* \omega_0$ on $(-\epsilon,\epsilon) \times\mathcal{U}$. This is a $b$-symplectic structure, invariant under the $S^1$-action agreeing with $\omega$ at $z$. By Theorem \ref{th:equivariantrelativeMoser} % and making $\mathcal{U}$ smaller if necessary,
there is an equivariant $b$-symplectomorphism $\varphi$ between neighbourhoods of $\mathcal{O}_z$ such that $\varphi^*(\psi^*\omega_0)=\omega$. Making $\mathcal{U}$ smaller if necessary and setting $\phi=\psi\circ\varphi$ we obtain the $b$-symplectomorphism given in the statement of the theorem.
\end{proof}

\begin{remark}
 Note that the modular period of the form $\omega_0$ is $\frac{k}{l} c$ where $c$ is the modular period of the $b$-symplectic form. {This is not necessarily the modular period of the original form $\omega$ cf. remark \ref{rem:modularperiod}.}
\end{remark}

\begin{example}
Consider the following symplectic mapping torus: take as a symplectic leaf a torus $\mathbb{T}^2$ with coordinates $(\varphi,\psi)$, $\varphi, \psi \in \mathbb{R} \,\mod\, 1$ equipped with the standard symplectic form and the holonomy map given by the diffeomorphism of $\mathbb{T}^2$ which descends from the diffeomorphism of $\mathbb{R}^2$ given by $\phi\in \textup{GL}(2,\mathbb{Z})$:
$$\phi=\left( \begin{array}{ccc}
0 & -1 \\
1 & 0  \\ \end{array} \right).$$
The mapping on $\mathbb{R}^2$ corresponds to rotation by $\frac{\pi}{2}$ and so we have $\phi^4=\text{Id}$. Denote the mapping torus $Z=([0,1] \times \mathbb{T}^2)/(0,x)\sim(1,\phi(x))$.

Consider the following $b$-symplectic form on $(t,\varphi,\psi,s) \in Z\times S^1$:
$$\omega=dt \wedge\frac{ds}{\sin(s)} + \beta$$
where $\beta$ is the standard symplectic form on $\mathbb{T}^2$. Consider the action of $S^1$ on $ Z\times S^1$ given by translation in the $t$-coordinate. Then there is a neighbourhood of a regular orbit contained in $Z$ which is equivariantly diffeomorphic to a neighbourhood of the zero section $(t, \textbf{0})$ of $S^1\times \mathbb{R}^3$ where $S^1$ acts by translations on the $S^1$ factor of  $S^1\times \mathbb{R}^3$. Moreover, there exist  coordinates $(t, x, y, a)$ on  $S^1\times \mathbb{R}^3$ so that the equivariant diffeomorphism becomes a symplectomorphism where $S^1\times  \mathbb{R}^3$ is equipped with the $b$-symplectic form
\begin{equation}\label{torusexample}
\omega=4 dt\wedge\frac{da}{a} + dx \wedge dy
\end{equation}
On the critical set there is also the exceptional orbit at $\phi=\psi=0$. In a neighbourhood of the singular orbit the $b$-sympletic form is the quotient of the $b$-symplectic structure \eqref{torusexample} given above where the group action $\sigma_n\in GL(2,\mathbb{Z})$ on the vector space $(x,y)$ is given by
$$\sigma_n=\left( \begin{array}{ccc}
0 & -1 \\
1 & 0  \\ \end{array} \right)^n.$$
\end{example}

\begin{example} We can find examples from integrable systems having a naturally associated $S^1$-action model with non-trivial isotropy group.

 Take $M=\,\! T^*S^1\times \mathbb{R}^2$ endowed with coordinates $(p,t, x, y)$  and  $b$-symplectic form $\omega=\frac{1}{p}dp\wedge dt + dx\wedge dy$. Consider the $b$-integrable system on $M$ given by $F=(\log(p), xy)$. This $b$-integrable system has hyperbolic singularities. Now let $\mathbb Z/2\mathbb Z$ act on $M$ in the following way: $(-1)\cdot(p,t, x, y)= (p, t, -x', -y')$ observe that this action leaves the hyperbola $xy=\text{cnst}$ invariant and switches its branches. The action clearly preserves the $b$-integrable system and induces a new integrable system on the quotient space $M/\!\sim$. Observe that the first component of the integrable system naturally induces an $S^1$-action given by the $b$-symplectic vector field associated to the singular Hamiltonian function $\log(p)$ (named as $b$-function, see \cite{guillemin2015toric} for a discussion). This circle action also descends to the quotient and the model for the circle action has non-trivial isotropy group of order two.

 This twisted hyperbolic case in $b$-symplectic manifolds is a reminiscent of the twisted hyperbolic construction in the symplectic case in \cite{curras2003symplectic} and \cite{miranda2004equivariant}\footnote{This example shows up in physical examples and corresponds to the 1:2 resonance (see for instance the example in page 32 of the monograph \cite{efstathiou2005metamorphoses})} and it is an invitation to study the invariants of a non-degenerate singularity of a $b$-symplectic manifold. This example can be extended to higher dimensions and the action of a $\mathbb Z/2\mathbb Z$  can be considered for every hyperbolic block added as long as the corank of the singularity is equal or bigger than one. The situation can be visualized using the \emph{curled torus}, the picture below showing the structure of the set $p=0, xy=0$.
\end{example}

\begin{figure}
\includegraphics[width = 10cm]{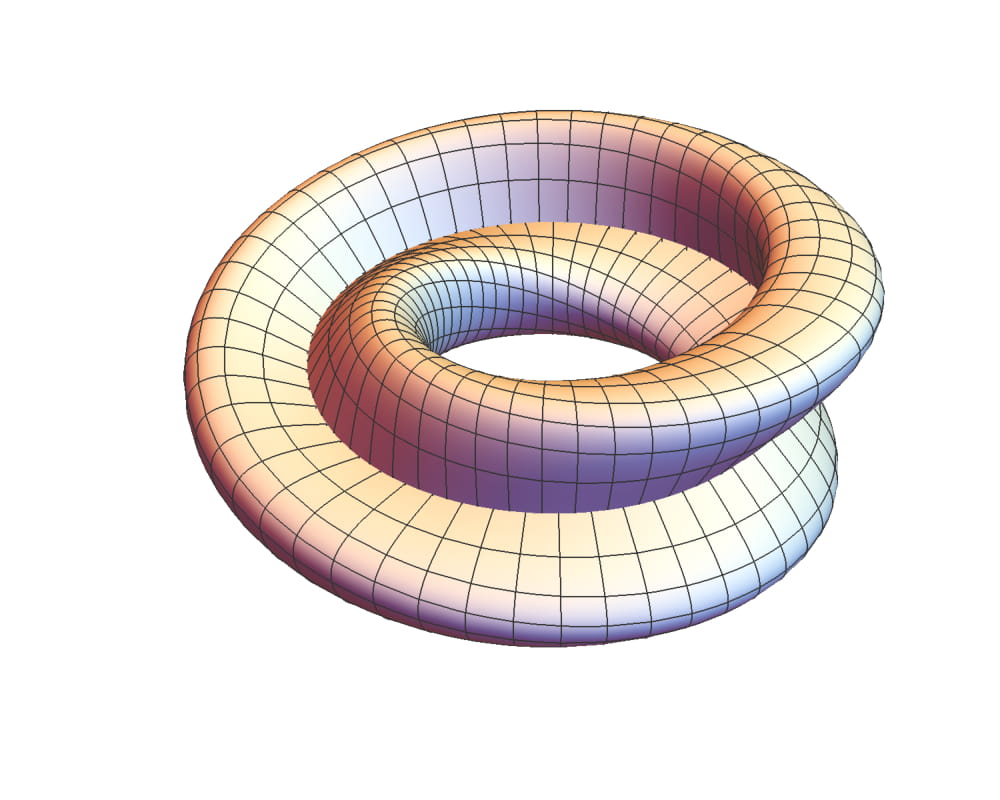}
\caption{The Curled Torus. 	Source: Konstantinos Efstathiou. }
\end{figure}

\section{Actions of compact Lie Groups on cosymplectic manifolds}

We treat the case of more general group actions on a $b$-symplectic manifold close to the critical set. First, we prove that only groups of a particular form can act on a $b$-symplectic manifold. For now we will treat group actions on a mapping torus $Z$ and then extend the results to a neighbourhood of the critical set.

In the following, we assume that the group $G$ is compact and connected and acts on a mapping torus $Z$ via a transverse, effective and foliation preserving action $\rho$. For a more general treatment of the lifting of group actions see \cite{montaldi2009symplectic}.

\begin{proposition}\label{prop:hfixesallleaves}
Suppose an element $h \in G$ fixes a leaf of the mapping torus, $\rho_h(\mathcal{L}_0)=\mathcal{L}_0$. Then $h$ fixes every leaf of the mapping torus.
\end{proposition}

This is an easy consequence of the fact that compact connected subgroups of $\text{Diffeo}^{+}(S^1)$ are conjugate to $\text{SO}(2)$, which is itself a consequence of $\text{Diffeo}^{+}(S^1)$ having a unique maximal compact subgroup, see \cite{ghys2001groups} for the case of orientation preserving homeomorphisms which can be adapted mutatis mutandis for the smooth case.

\begin{proof}
Let $\pi:Z \to S^1$ be the mapping torus projection. The action of the group $G$ on a symplectic mapping torus $Z$ induces an action of $G$ on the base $S^1$ in the obvious fashion
\begin{align}\label{inducedaction}
\tau:G\times S^1 &\rightarrow S^1 \\
(g,\pi(x)) &\mapsto \pi(\rho_g( x)) =: \tau_g(x)
\end{align}
As $G$ is compact and connected its image $\tau(G,\cdot)$ is a compact subgroup of $\text{Diffeo}^{+}(S^1)$, the group of orientation preserving diffeomorphisms of the circle. Whence $\tau(G,\cdot)$ is conjugate by some $w\in \text{Diffeo}^{+}(S^1)$ to $\text{SO}(2)$.
Suppose $h\in G$ fixes a leaf $\mathcal{L}_0$. This corresponds to a fixed point of the induced action $\tau_h$ on $S^1$, and so a fixed point for $w\tau_h w^{-1}\in \text{SO}(2)$. Whence $w \tau_h w^{-1}=\text{Id}_{S^1}$ and so $\tau_h=\text{Id}_{S^1}$. This corresponds to $h$ fixing all leaves of $Z$.
\end{proof}

It can be checked easily that this defines a subgroup of $G$. We  call $$H_0=\{h \in G\,|\,\rho_h(\mathcal{L}_0)=\mathcal{L}_0\}$$
the \emph{leaf preserving} subgroup of $G$.

\begin{proposition}\label{prop:H0}
Let $G$ be a group acting in a transverse and foliation preserving manner on a symplectic mapping torus. Let $H_0$ be the leaf preserving subgroup of $G$. Then \begin{enumerate}[(i)]
    \item  $H_0$ is a normal subgroup of $G$.
    \item $H_0$ is a closed Lie subgroup of $G$.
    \item The codimension of $H_0$ in $G$ is one.
\end{enumerate}
\end{proposition}

\begin{proof}

\begin{enumerate}[(i)]
\item This follows immediately from the fact that for $h\in H_0$, $g\in G$ we have $\tau_{ghg^{-1}}=\tau_g \tau_h \tau_g^{-1} = \tau_g \tau_g^{-1}=\text{Id}_{S^1}$, hence $ghg^{-1} \in H_0$.
\item Consider the projection
\begin{equation*}
\begin{split}
\Phi:G &\rightarrow \text{SO}(2)\\
\Phi(g)&=w \tau_g w^{-1}
\end{split}
\end{equation*}
corresponding to the map from $G$ to $\text{SO}(2)$ given in Proposition \ref{prop:hfixesallleaves}. It is clear that the level set $\Phi^{-1}(\text{Id})$ consists precisely of the elements of $G$ which are leaf preserving. Hence $\Phi^{-1}(\text{Id})=H_0$ is a closed subgroup of $G$.
\item The codimension of $H_0$ is at most one since it is given as the level set $\Phi^{-1}(\text{Id})=H_0$. As $G$ induces an action transverse to the foliation of $Z$ it follows that the codimension of $H_0$ is exactly one.
\end{enumerate}
\end{proof}

\begin{proposition}\label{prop:productaction}
The action of $G$ on the mapping torus $Z$ lifts to an action of a product group $\tilde{G}=S^1\times H$ on a finite trivializing cover of $Z$ where $H$ is compact and connected. Moreover, $G$ is necessarily of the form $G=(S^1 \times H)/\Gamma$ for a finite cyclic subgroup $\Gamma$ (which might be trivial).
\end{proposition}

\begin{proof}
Let $\mathfrak{h} \subset \mathfrak{g}$ be the Lie algebra of $H_0$, the leaf preserving subgroup of $G$ and consider a complementary ideal $\mathfrak{k}$ of $\mathfrak{h}$ in $\mathfrak{g}$. {From the construction in the proof of proposition \ref{prop:H0} we can check that the subgroup $K=\exp(\mathfrak{k})$ is closed (and indeed isomorphic to $SO(2)\cong S^1)$.}
%$K$ is then a one dimensional closed subgroup of the compact group $G$ and so $K \cong S^1$.
The action of $K$ is transverse to the foliation and so by Proposition \ref{finitecover} there exists a finite trivializing cover $\tilde{Z} \cong S^1 \times \mathcal{L}$ of $Z$, such that $Z$ is the quotient of $\tilde{Z}$ by the action of the leaf fixing subgroup $\Gamma' \cong \mathbb{Z}_k$ of $K$ on $\tilde{Z}$ where $\Gamma'$ acts as
\begin{equation}\label{propproductaction:mu}
\mu_m(t,l)=(t-\frac{m}{k},\sigma_{m}(l)),\quad m\in \Z_k, (t,l)\in S^1 \times \mathcal{L}
\end{equation}
and $\sigma$ is the leaf automorphism induced by the leaf-fixing elements of $K$ on $\mathcal{L}$. Denote $\exp(\mathfrak{h}) \subset G$ by $H$. Denote by $\tilde{G}$ the group $K \times H$. Then we have an action  $\tilde{\rho}$ of $\tilde{G}$ on $\tilde{Z}$ given by
$$\tilde \rho: \tilde{G}\times \tilde{Z} \to \tilde{Z},\quad  \tilde \rho_{(s,h)}(t,l)=(t+s,\rho_h(l)).$$
Suppose $\sigma_m=\rho_h$ as an equality of maps on $\mathcal{L}$ for some $h\in H,\, m \in \Gamma' \backslash\{0\}$. As $H$ is connected, $\sigma_1=\rho_{h'}$ for some $h'\in H$. Whence, the action $\mu$ of $\Gamma'$ on $\tilde{Z}$ is equivalent to the action $\tilde{\rho}$ of $\Gamma\subset \tilde{G}$ on $\tilde{Z}$ where $\Gamma$ is the group
$$\Gamma=\left\{\Big(-\frac{m}{k}, (h')^{m}\Big)\,\Big| \, m = 0, \ldots, k-1 \right\},$$
i.e. $\mu_m = \tilde \rho_{(-\frac{m}{k},(h')^m)}$ for all $m\in \mathbb{Z}_k$
Letting $p_{\tilde{Z}}$ and $p_{\tilde{G}}$ denote the projections to $\tilde{Z}/\Gamma$ and $\tilde{G}/\Gamma$  respectively, we have a commutative diagram
\begin{align*}
\begin{diagram}
\node{\tilde{G} \times \tilde{Z}}\arrow{e,t}{\tilde{\rho}}
\arrow{s,l}{p_{\tilde{G}} \times p_{\tilde{Z}}}
\node{\tilde{Z}} \arrow{s,r}{p_{\tilde{Z}}}\\
\node{\tilde{G}/\Gamma \times Z}  \arrow{e,t}{\rho}  \node{Z}
\end{diagram}
\end{align*}
By construction, the action of $\tilde{G}/\Gamma$ on $Z$ and the action of $G$ on $Z$ possess the same fundamental vector fields. Moreover, the action of both groups is effective. Necessarily, then, $\tilde{G}/\Gamma=G$.
Conversely, assume that $\sigma_{1} \neq \rho_{h}$ for any $h \in H$.

Then $\exp(\mathfrak{k}) \cap \exp(\mathfrak{h})=0$ and so $G \cong K\times H$. The action $\tilde{\rho}$ of $\tilde{G} \cong\ G$ on $\tilde{Z}$ projects to an action of $G$ on $Z$ where the projection $p_{\tilde{G}}$ is given by quotienting the group by the subgroup $\Gamma=\Gamma'\times \{ e_{H} \}$ and the projection $p_{\tilde{Z}}$ is given by the action $\mu$ of $\Gamma\cong\Gamma'$ in \eqref{propproductaction:mu}.
\end{proof}

\begin{proposition}\label{prop:productaction2}
Let $G=S^1 \times H$ be a product group acting on a mapping torus $Z$ such that the $S^1$ factor acts transverse to the foliation. Let $z\in Z$ and denote by $G_z$ the isotropy group of $z$. Then $G_z \cong \mathbb{Z}_l \times H_{z}$ where $H_z$ the isotropy group of $z$ under the $H$-action and $\mathbb{Z}_l$ is a cyclic subgroup and $\mathbb{Z}_l\times\{e_H\}$ acts as the identity on $\mathcal{O}^H_{z}$.
\end{proposition}

\begin{proof}
Let $\mathcal{L}_0$ be a leaf of $Z$. % and as before denote by $H$ the subgroup $(0, \exp(\mathfrak{h})) \subset G$,
Denote by $\mathcal{O}^{H}_z \subset \mathcal{L}_0$ the orbit of $z$ under the action of $(0,H)\subset G$. Denote the subgroup $(S^1, e_{H})\subset G$ by $K$. Let $\rho^{K}$ be the action of $K$ on $Z$. The  leaf preserving subgroup of $K$ can be identified with $\mathbb{Z}_k$; for $m\in \mathbb{Z}_k$, i.e. $\frac{m}{k}\in K$ is leaf-preserving, we either have $\rho^{K}_{\frac{m}{k}}(\mathcal{O}^{H}_z)\cap \mathcal{O}^{H}_z=\emptyset$ or $\rho^{K}_{\frac{m}{k}}(\mathcal{O}^{H}_z)\cap \mathcal{O}^{H}_z=\mathcal{O}^{H}_z$. Moreover elements $m\in \mathbb{Z}_k$ satisfying $\rho^{K}_{\frac{m}{k}}(\mathcal{O}^{H}_z)\cap \mathcal{O}^{H}_z=\mathcal{O}^{H}_z$ form a subgroup $\mathbb{Z}_l$ of $\mathbb{Z}_k$.

If $\rho^{K}_{\frac{m}{k}}(z) \notin \mathcal{O}^{H}_{z}$ for all $m\in \mathbb{Z}_k$ then $\mathbb{Z}_l=\{0\}$ and $G_z=\{0\} \times H_z$ where $H_z$ is the isotropy group of $z$ under the action of $(0, H)$. Alternatively suppose $\mathbb{Z}_l\neq \{0\}$, so that $\rho^{K}_{\frac{1}{l}}(z)=h \cdot z$ for some $h\in H$. If $h\neq e_{H}$, we can find a new $K' \subset G$ which acts as the identity on $z$ as follows: let $\eta$ in $\mathfrak{k}$ be such that $K=\exp(t\eta)$ where $t\in[0, 1)$. Let $\nu\in \mathfrak{h}$ be such that $\exp(\frac{1}{l}\nu)=h^{-1}$. Consider the subgroup
$$K'=\{\exp(t(\eta+\nu)) |\, t\in[0,1)\}.$$
Then the isotropy group of $z$ is of the form $\mathbb{Z}_l \times H_z$ where
$\mathbb{Z}_l \cong \{\exp\left(\frac{n}{l}(\eta+\nu)\right) |\, n=0,\ldots,l-1\}$.
\end{proof}

\section{A $b$-symplectic slice theorem}

Let $(M \cong (-\epsilon,\epsilon)\times Z, \omega)$ be a $b$-symplectic manifold together with an effective $b$-symplectic action by a compact connected Lie group $G$ acting transversely to the symplectic leaves inside the critical hypersurface $Z$.
First we will construct the $b$-symplectic models which will give us a normal form for the $b$-symplectic form about an orbit of $G$. By Proposition \ref{prop:productaction} there are two distinct cases
\begin{enumerate}
    \item $G$ is a group isomorphic to the product of Lie groups $G=S^1\times H$.
    \item $G=(S^1\times H)/\Gamma$ where $\Gamma\cong\mathbb{Z}_l\times\mathbb{Z}_k\subset S^1\times H$ and $\mathbb{Z}_k$ is a non-trivial cyclic subgroup of $H$.
\end{enumerate}

Recall from Proposition \ref{prop:productaction} that there is a trivial finite cover $\tilde{Z}= S^1\times \mathcal{L}$ of $Z$ equipped with an $(S^1 \times H)$-action which projects to the action of $G$ on $Z$.

Let $z\in \mathcal{L}_0$ be a point in a symplectic leaf of $Z$ and consider the orbit $\mathcal{O}^{H}_z$ of $z$ given by the group action of $H=\exp{\mathfrak{h}}$ on $\mathcal{L}_0$. % and the symplectic form induced on $\mathcal{L}_0$.
Denote the isotropy group of $z$ by $H_z$. By the symplectic slice theorem (Theorem \ref{symplecticslicetheorem}), there is an $H$-equivariant neighbourhood $\mathcal{U}$ of $\mathcal{O}^{H}_z$ in $\mathcal{L}_0$ which is equivariantly symplectomorphic to a neighbourhood of the zero section of the vector bundle $Y_z^H = (H \times \mathfrak{m}^*\times V_z)/H_z$ with symplectic form $\omega_{MGS}$ as given by Theorem \ref{symplecticslicetheorem}. Recall that $\mathfrak{m}$ is a  Lie subalgebra of $\mathfrak{h}$, the vector space $V_z\subset T_z\mathcal{L}_{0}$ is the symplectic orthogonal $V_z=(T_{m} \mathcal{O}_z^{H})^{\omega} / T_{m} \mathcal{O}_z^{H}$ and $H_z$ acts on $V_z$ by the isotropy representation.

\begin{definition}[$b$-Symplectic models]\label{bsymplecticmodel}
Consider the $b$-symplectic form on $\tilde{E}=T^*S^1\times(H\times_{H_z} \mathfrak{m}^*\times V_z)$  given by
\begin{equation}\label{coverform}
\tilde \omega_0=\omega_{c'}+\omega_{MGS}
\end{equation}
where $\omega_{c'}$ is the $b$-symplectic form on $T^*S^1$ of modular period $c'=ck$  given by Definition \ref{twistedsymplectic},
and $\omega_{MGS}$ is the symplectic form on $Y_z^{H} = H\times_{H_z} \mathfrak{m}^*\times V_z$ given by the symplectic slice theorem (Theorem \ref{symplecticslicetheorem}). Consider the quotient $b$-symplectic structure on $E=\tilde{E}/\mathbb{Z}_l$ where $ \frac{m}{l}\in \mathbb{Z}_l$ acts on $T^{*}S^1$ as the cotangent lift of $\mathbb{Z}_l$ acting by translations on $S^1$ and acts on the factor $H \times_{H_z} \mathfrak{m}^*\times V_z$ equipped with the coordinates $[k, \eta, v]$ of Theorem \ref{symplecticslicetheorem} either by
\begin{enumerate}
\item  $\frac{m}{l} \cdot  [k, \eta, v]=[k, \sigma^m(\eta), \sigma^m(v)]$ for a linear symplectomorphism $\sigma$.
\item $\frac{m}{l} \cdot [k, \eta, v]=[h^m \cdot k, \eta, v]$ where $h$ is some element of $H$.
\end{enumerate}
Then $E$ has a unique $b$-symplectic structure such that the projection is a local $b$-symplectomorphism (see Remark \ref{remark:quotientpoisson}). We call these normal forms \textbf{$b$-symplectic models with symplectic slice $V_z$ and modular period $c\frac{k}{l}$}.
\end{definition}

We will now show that a neighborhood of an orbit of a $G$-action on a $b$-symplectic manifold with the properties stated in the beginning is $b$-symplectomorphic to a neighborhood of the zero section of one of the models above, completing the proof of Theorem \ref{bslicetheorem}, which we recall here in a succinct way:

\begin{theorem}
Let $G$ be a compact Lie group acting on a $b$-symplectic manifold $(M,\omega)$ transverse to the symplectic foliation { and that there is one symplectic leaf of the critical set $Z$ which is compact}. Suppose that the action of $G$ is $b$-symplectic, effective and Hamiltonian when restricted \footnote{{The group action $\rho$ induces a Lie algebra homomorphism from $\mathfrak{g}$ to the Lie algebra of vector fields on $M$ via $d_{e} \rho_{x}: \mathfrak{g} \rightarrow T_{x} M$.  In turn, by restricting the action to the critical set, we have a Lie algebra homomorphism $d_{e} \rho_{x}\vert_Z: \mathfrak{g} \rightarrow T_{x} Z \cong T_{x} S^1 \times T_{x} \mathcal{L}$ where $\mathcal{L}$ is a leaf of the foliation. The action is transverse to the leaf of the foliation if the restriction of this homomorphism to the first factor, $T_{x} S^1$, is surjective everywhere. The action is Hamiltonian when restricted to leaves if the vector fields given by the restricted homomorphism $d_{e} \rho_{x}\vert_\mathcal{L}: \mathfrak{g} \rightarrow T_{x} \mathcal{L}$ are Hamiltonian.}} to a symplectic leaf of $Z$.   Let $\mathcal{O}_z$ be an orbit of the group action contained in the critical set of $M$. Then there is a neighbourhood $\mathcal{V}$ of $\mathcal{O}_z$ in $M$ which is $b$-symplectomorphic to a neighbourhood of the zero section of a bundle given by the $b$-symplectic model $E$ in Definition \ref{bsymplecticmodel}.
\end{theorem}

\begin{proof}
By Proposition \ref{prop:productaction} there exists a finite cover $\tilde{Z}\cong S^1 \times \mathcal{L}$ which comes equipped with the action of a product group $\tilde{G}\cong S^1 \times H$ which covers the action of $G$ on $Z$. Let $z\in \mathcal{L}_0 \subset Z$ and let $\tilde{z}\in \tilde{Z}$ be a point projecting to $z$. Denote by $\mathcal{O}^{H}_{{z}}$  the orbit of {$z$} under the action of the subgroup $H$ and by $\mathcal{O}^{\tilde{G}}_{\tilde{z}}$ the orbit of $\tilde{z}$ under the action of $\tilde{G} \cong S^{1} \times H$ on the cover $\tilde{Z}$. Then an invariant open neighbourhood of $\mathcal{O}^{\tilde{G}}_{\tilde{z}}$ in $\tilde M \cong (-\epsilon,\epsilon) \times \tilde Z$ is of the form $\tilde{\mathcal{V}}=(-\epsilon,\epsilon) \times S^1  \times \mathcal{U}$ where $\mathcal{U}$ is an invariant open neighbourhood of $\mathcal{O}^{H}_{{z}}$ in $\mathcal{L}_0$. Recall that $Z$ is the quotient of $\tilde{Z}$ by a cyclic subgroup $\Gamma$ of $\tilde{G}$. By Proposition \ref{prop:productaction}, we may assume that $\Gamma$ is of the form
\begin{equation}\label{Gamma}
\Gamma=\left\{\left(-\frac{m}{k}, h^{m}\right) \big| \,m = 0, \ldots, k-1 \right\}.
\end{equation}
Let $\tilde{\omega}$ be the lift of $\omega$ to $\tilde{Z}$ as given by Proposition \ref{finitecover}. By Theorem \ref{theorem:simplyconnected} we may assume that locally around $\tilde z$, $\tilde{\omega}$  is of the form
$$\tilde{\omega}=ck d t \wedge \frac{d a}{a}+\beta$$
where $\beta$ is the symplectic form on the leaf $\mathcal{L}_0$. Denote by $H_{{z}}$ the isotropy group of ${z}$ under the action of $H$.
By the symplectic slice theorem, Theorem \ref{symplecticslicetheorem}, a neighbourhood $\mathcal{U}$  of $\mathcal{O}^{H}_{z}$ with symplectic form $\beta$ is equivariantly symplectomorphic to a neighbourhood of the zero section of the bundle $Y^{H}_{z}=H \times_{H_{z}} \mathfrak{m}^* \times V_{z}$ with symplectic form given by the symplectic slice theorem.

Consider the vector bundle $\tilde{E}=T^{*} S^{1} \times\left(H \times_{H_{z}} \mathfrak{m}^{*} \times V_{z}\right)$ with symplectic form given by $\tilde{\omega}_0$ in Equation \eqref{coverform}, where $c$ is the modular period of $\omega$ and $k$ is the order of $\Gamma$. Let $\tilde{\psi}$ be the equivariant diffeomorphism form $\tilde{\mathcal{V}}$ to a neighborhood of the zero section in $\tilde E$ obtained from the slice theorem on $\mathcal{L}_0$. Then $\tilde{\psi}^{*} \tilde{\omega}_0$ is a $b$-symplectic form on a neighbourhood of $\mathcal{O}^{\tilde{G}}_{\tilde{z}}$ and, since $(\tilde{\psi}^{*}\tilde{\omega}_0)_{\tilde{z}}=\tilde{\omega}_{\tilde{z}}$, by the equivariant relative Moser Theorem, Theorem \ref{th:equivariantrelativeMoser}, and after making $ \mathcal{\tilde V}$ smaller if necessary, we can conclude that there is an equivariant $b$-symplectomorphism from $\mathcal{\tilde V}$ to a neighbourhood of the zero section of $\tilde{E}$ equipped with the $b$-symplectic form $\tilde{\omega}_0$.

%Denote the action of $m \in \Gamma$ by $\tilde{\rho}_m$.
%If for all $m\in \Gamma, m \neq 0$, $\mu_m(\mathcal{O}^{H}_z) \cap \mathcal{O}^{H}_z=\emptyset$ then the action $\tilde \rho$ on $\mathcal{\tilde V}$?} is free and, shrinking the neighbourhood if necessary, the projection $p:\tilde{Z}\rightarrow Z$ restricts to an equivariant symplectomorphism from a neighbourhood $\tilde{\mathcal{V}}$ of $\mathcal{O}_{\tilde{z}}^{\tilde{G}}$ to a neighbourhood $\mathcal{V}$ of $\mathcal{O}_z^G$. {I guess this is a matter of taste, but I wouldn't even have considered this case separately, as it is a special case of the next paragraph with $\Gamma_z=0$.}
Denote by $\Gamma_z$ the subgroup of $\Gamma$ defined by $\{m \in\Gamma\,|\,\rho_m(\mathcal{O}^H_z) \cap \mathcal{O}^H_z=\mathcal{O}^H_z\}$, where $\rho$ is the action of $\Gamma_z \subset \tilde{G}$ on $Z$ equivariant with respect to the projection $\tilde Z \to Z$, as in Proposition \ref{prop:productaction}.

Then $\Gamma_z$ is a cyclic subgroup $\Gamma_z \cong \mathbb{Z}_l$ of $\Gamma$ of the form
$$\Gamma_z=\left\{\left(-\frac{m}{k}, (h')^m\right) \big|\, m = 0, \ldots, l-1 \right\}$$
for some $h' \in H$. Denote by $p_{\tilde{\mathcal{V}}}$, $p_{\tilde{E}}$ the projections to the quotients of $\tilde{\mathcal{V}}$ and $\tilde{E}$ by the action of $\Gamma_z$ respectively. Define $\psi$ by the condition that the following diagram commutes:

\begin{align*}
\begin{diagram}
\node{\tilde{\mathcal{V}}}\arrow{e,t}{\tilde{\psi}}
\arrow{s,l}{p_{\tilde{\mathcal{V}}}}
\node{\tilde{E}} \arrow{s,r}{p_{\tilde{E}}}\\
\node{\tilde{\mathcal{V}}/\Gamma}  \arrow{e,t}{\psi}  \node{E}
\end{diagram}
\end{align*}

First consider the case where $G\cong S^1\times H$. We may assume by Proposition \ref{prop:productaction2} that the action of $\Gamma_z$ on the orbit $\mathcal{O}^{H}_z$ and so on the base of the bundle $Y^{H}_{z}$ is trivial. Moreover it preserves the slice $V_z$ and acts by linear symplectomorphisms and so $\psi$ is $b$-symplectomorphism to Model (1) of Definition \ref{bsymplecticmodel}.

For $h\neq e_H$ in the group $\Gamma$ in Equation \ref{Gamma} (that is, the case $G\cong (S^1 \times H)/\Gamma$ for $\Gamma$ non trivial and $\Gamma_z\subset\Gamma$), the the action of $\Gamma_z$ on $Y^{H}_{z}$ is given by the symplectic slice theorem, Theorem \ref{symplecticslicetheorem}, and the equivariant normal form is given by Model (2) of Definition \ref{bsymplecticmodel}.
\end{proof}

{ Some remarks are in order: }\begin{remark}  {It would be possible to extend this slice theorem to proper group actions rather than compact group actions as done in \cite{ortega2002symplectic}.}

\end{remark}

\begin{remark} {This slice theorem puts a first step forwards towards understanding rigidity phenomena in the Poisson realm. In \cite{mmz} a rigidity theorem for Hamiltonian actions on Poisson manifolds is given. Our slice theorem yields a similar rigidity theorem also for $b$-symplectic actions. The general rigidity theorem in \cite{mmz} uses sophisticated Nash-Moser techniques which are not necessary in the case of $b$-symplectic manifolds. In a more general context, this is connected to the problem of stability of symplectic leaves (the normal form obtained gives an equivariant version of this phenomenon). For the general problem of stability of symplectic leaves and rigidity and flexibility phenomena confer to \cite{stab} and \cite{rui2}.}

\end{remark}

\begin{example}
Let $G$ be a compact Lie group with non-trivial centre. Let $\xi_{1} \in \mathfrak{g}$ be a central element of the Lie algebra and $\xi_{2}, \ldots, \xi_{n}\in \mathfrak{g}$ be such that $\xi_{1}, \ldots, \xi_{n}$ form a basis of the Lie algebra. Denote by $\eta_{i}$ the basis of the Lie algebra dual such that $\left(\eta_{i}, \xi_{j}\right)=\delta_{i j} .$ Denote the associated invariant vector fields $L_{g *} \xi_{i}$ by $v_{i}$ and $L_{g}^{*} \eta_{i}$ by $m_{i}$ respectively. At each point $g \in G$ these give a basis for the tangent and cotangent spaces at $g$. Consider the singular 2 -form on $T^{*} G$
\begin{equation}
\omega=\pi^*m_{1} \wedge \frac{d\left(\lambda\left(v_{1}\right)\right)}{\lambda\left(v_{1}\right)}+\sum_{i=2}^{n} \pi^*m_{i} \wedge d\left(\lambda\left(v_{i}\right)\right)
\end{equation}
where $\pi$ the canonical projection $T^{*} G \rightarrow G$. It can be checked directly that $\omega$ is a $b$-symplectic form on $T^*G$ invariant under the cotangent lifted action of $G$ on $T^*G$. By Proposition \ref{prop:productaction}, $G$ has a finite cover of the form $S^1\times H$. The $b$-symplectic model for the action of $G$ on $T^{*} G$ is given by $\tilde{E}=(T^{*} S^{1} \times T^*H)/\mathbb{Z}_k$ where $\mathbb{Z}_k$ acts diagonally on $T^{*} S^{1} \times T^*H$ by the cotangent lift of translations on $S^1$ and the $b$-symplectic form on $\tilde E$ is
\begin{equation}
\tilde \omega_0 = \omega_{c}+\omega_{H}
\end{equation}
where
\begin{itemize}
\item $\omega_{c}$ is the standard $b$-symplectic form of modular period $c$ on the manifold $T^*S^1$, as given in Definition \ref{twistedsymplectic}.
\item $\omega_H$ is the canonical symplectic form on $T^{*}H$.
\end{itemize}
\end{example}

\begin{example}
Consider the symplectic mapping torus
\begin{equation}
Z=\frac{[0, 1] \times \mathcal{L}}{(0,l) \sim(1,\phi(l))}
\end{equation}
where
\begin{itemize}
    \item $\mathcal{L}\cong S^2\times S^2$, where $S^{2}$ is the two sphere equipped with the standard symplectic form and $\mathcal{L}$ is equipped with the product symplectic form.
    \item $\phi: \mathcal{L}\rightarrow \mathcal{L}$ is the diffeomorphism of $\mathcal{L}$ given by exchanging the $S^2$ factors of $\mathcal{L}$, i.e., $\phi(x,y)=(y,x)$.
\end{itemize}
Consider the group $G = S^1\times \text{SO}(3)$ where $\text{SO}(3)$ acts diagonally on the product $\mathcal{L}\cong S^2\times S^2$ by rotation on each factor and $S^1$ acts by translations on the factor $[0,1]$ of the above mapping torus:
$$(s,A) \cdot (t,x,y) = (t+ 2s, A \cdot x, A \cdot y), \qquad (s,A) \in S^1 \times \text{SO}(3),\, (t,x,y) \in [0,1]\times \mathcal{L}. $$
Consider a point $z = (0, x,y) \in Z$ and the corresponding orbit $\mathcal{O}_z$ in the $b$-symplectic manifold $M = (-\epsilon,\epsilon)\times Z$. We distinguish three cases:

First, suppose $x \neq \pm y$. Then the action of $S^1 \times \text{SO}(3)$ on the orbit $\mathcal{O}_z$ is free. There is a neighbourhood a neighbourhood $\mathcal{V}$ of $\mathcal{O}_z$ equivariantly $b$-symplectomorphic to the zero section of the bundle $E=T^{*} \mathrm{S}^{1} \times Y_{z}^{\text{SO}(3)}$, where $E$ is equipped with the $b$-symplectic form
\begin{equation*}
\tilde \omega_0 =\omega_{2}+ \omega_{M G S}
\end{equation*}
and $\omega_{2}$ is the standard symplectic form  of modular period 2 on $T^* S^1$ and $\omega_{M G S}$ is a symplectic form on $Y_{z}^{\text{SO}(3)}$ given by Theorem \ref{symplecticslicetheorem}.

Now let  $x=y$. Then $z$ has isotropy group $\mathbb{Z}_2\times \text{SO}(2)$. The associated $b$-symplectic model is given by $E=T^{*} \mathrm{S}^{1} \times F$, where $F=\text{SO}(3) \times_{\text{SO}(2)} V$ is a bundle over the homogeneous space $\text{SO}(3)/\text{SO}(2)\cong S^2$, $V$ a $2$-dimensional vector space with Darboux symplectic form $\omega_V$. The $b$-symplectic form on $E$ is given by
$$\tilde \omega_0=\omega_{1}+2\omega_{S^2}+\omega_V$$
where $\omega_{1}$ is the standard symplectic form of modular period 1 on $T^* S^1$  and $\omega_{S^2}$ is the usual symplectic form on the sphere.% and $\omega_V$ the linear Darboux form on $V$.

Finally, suppose $y=-x$. Let $\nu\in \mathfrak{k}$, $\mathfrak{k}$ the Lie algebra of $S^1$. Let $\exp(t\xi) \cong \text{SO}(2)$ be the $1$-parameter subgroup of $SO(3)$ such that $g=exp(\xi)$ acts on $S^2$ by $g(x)=-x$ and $d\rho_g=-Id$. Then the subgroup $(exp(t\nu),exp(t\xi)) \hookrightarrow S^1\times SO(3)$ acts as the identity on the orbit $\mathcal{O}_z\cong S^2$ and the $b$-symplectic model is given by the quotient bundle $E=T^{*} \mathrm{S}^{1} \times (SO(3) \times_{SO(2)} V)/\mathbb{Z}_2$ where $\mathbb{Z}_2$ acts on $u,v \in V$ by $(u,v)\rightarrow (-u,-v)$ and $E$ is equipped with the $b$-symplectic form
\begin{equation*}
\tilde \omega_0=\omega_{1}+2\omega_{S^2}+\omega_V.
\end{equation*}
\end{example}


\begin{thebibliography}{10}

\bibitem{bazzoni2014structure}
Giovanni Bazzoni and John Oprea,
\newblock On the structure of co-k{\"a}hler manifolds,
\newblock {\em Geometriae Dedicata}, 170(1):71--85, 2014.


\bibitem{bochner} S. \ Bochner,
\newblock{Compact groups of differentiable
transformations,}
\newblock {\em Ann. of Math.}  46(2), 372--3817, 1945.

\bibitem{besse}
A. Besse, {\it Manifolds all of whose geodesics are closed}, Vol. 93. Springer Science and Business Media, 2012.



\bibitem{bolsinov}Alexey V. Bolsinov, Bozidar Jovanovic, Non-commutative integrability, moment map and geodesic flows, \emph{Annals of Global Analysis and Geometry} 23(4), 305--322, 2003.


\bibitem{rui2} M. Crainic and R. L. Fernandes, \emph{Rigidity and flexibility in Poisson geometry}, Trav. Math., 16 (2005), 53-68.

\bibitem{stab} M. Crainic and R. L. Fernandes, \emph{Stability of symplectic leaves.
 }, Invent. Math. 180 (2010), no. 3, 481–533.

\bibitem{curras2003symplectic}
Carlos Curr{\'a}s-Bosch and Eva Miranda.
\newblock Symplectic linearization of singular lagrangian foliations in $M^4$.
\newblock {\em Differential Geometry and its applications}, 18(2), 195--205, 2003.

\bibitem{efstathiou2005metamorphoses}
Konstantinos Efstathiou.
\newblock {\em Metamorphoses of Hamiltonian systems with symmetries}.
\newblock Springer, 2005.

\bibitem{ghys2001groups}
{\'E}tienne Ghys.
\newblock Groups acting on the circle.
\newblock {\em Enseignement Mathematique}, 47(3/4), 329--408, 2001.

\bibitem{gualtieri2017tropical}
Marco Gualtieri, Songhao Li, Alvaro Pelayo, and Tudor~S Ratiu.
\newblock The tropical momentum map: a classification of toric log symplectic
  manifolds.
\newblock {\em Mathematische Annalen}, 367(3-4):1217--1258, 2017.

\bibitem{guillemin2011codimension}
Victor Guillemin, Eva Miranda, and Ana~Rita Pires.
\newblock Codimension one symplectic foliations and regular Poisson structures.
\newblock {\em Bulletin of the Brazilian Mathematical Society, New Series},
  42(4):607--623, 2011.

\bibitem{guillemin2014symplectic}
Victor Guillemin, Eva Miranda, and Ana~Rita Pires.
\newblock Symplectic and Poisson geometry on b-manifolds.
\newblock {\em Advances in mathematics}, 264:864--896, 2014.

\bibitem{guillemin2015toric}
Victor Guillemin, Eva Miranda, Ana~Rita Pires, and Geoffrey Scott.
\newblock Toric actions on b-symplectic manifolds.
\newblock {\em International Mathematics Research Notices},
  2015(14):5818--5848, 2015.


\bibitem{GMPS2}
V.\ Guillemin, E.\ Miranda, A.\ R.\ Pires and G.\ Scott,
\emph{Convexity for Hamiltonian torus actions on $b$-symplectic manifolds},
 Math. Res. Lett. 24 (2017), no. 2, 363–377.



\bibitem{GMWconvexity}
V.\ Guillemin, E.\ Miranda and J.\ Weitsman,
\emph{Convexity of the moment map image for torus actions on $b^m$-symplectic manifolds}
Philos.\ Trans.\ Roy.\ Soc.\ A 376 (2018), no.\ 2131

\bibitem{GMWgeomq}
V.\ Guillemin, E.\ Miranda and J.\ Weitsman,
\emph{On geometric quantization $b$-symplectic manifolds},
Adv.\ Math.\ 331 (2018), 941–951.


\bibitem{guillemin1990symplectic}
Victor Guillemin and Shlomo Sternberg.
\newblock {\em Symplectic techniques in physics}.
\newblock Cambridge university press, 1990.

\bibitem{kiesenhofer2017cotangent}
Anna Kiesenhofer and Eva Miranda.
\newblock Cotangent models for integrable systems.
\newblock {\em Communications in Mathematical Physics}, 350(3):1123--1145, 2017.


\bibitem{kiesenhofermiranda} A. \ Kiesenhofer and  E. \ Miranda, \emph{ Noncommutative integrable systems on b-symplectic manifolds.}  Regul. Chaotic Dyn. 21, no. 6, 643–659, 2016.

\bibitem{libermann1959automorphismes}
Paulette Libermann.
\newblock Sur les automorphismes infinitesimaux des structures syplectiques et
  des structures de contact.
\newblock {\em Coll. G{\'e}om. Diff. Globale}, 37--58, 1959.

\bibitem{marle1985modele}
Charles-Michel Marle.
\newblock Modele d'action hamiltonienne d'un groupe de lie sur une
  vari{\'e}t{\'e} symplectique.
\newblock {\em Rend. Sem. Mat. Univ. Politec. Torino}, 43(2):227--251, 1985.

\bibitem{miranda2004equivariant}
Eva Miranda and Nguyen~Tien Zung.
\newblock Equivariant normal form for nondegenerate singular orbits of
  integrable hamiltonian systems.
\newblock {\em Annales scientifiques de l'Ecole normale sup{\'e}rieure},
  volume~37, 819--839. Elsevier, 2004.


\bibitem{ortega2002symplectic}
Juan-Pablo Ortega and Tudor~S Ratiu.
\newblock A symplectic slice theorem.
\newblock {\em Lett Math. Phys. 59}, 81--93, 2002.

\bibitem{palais1}
Richard Palais,
\newblock The classification of G-spaces.
\newblock {\em Mem. Amer. Math. Soc. No.} 36, 1960.

\bibitem{palais2}
Richard Palais,
\newblock On the existence of slices for actions of non-compact Lie groups.
\newblock {\em Ann. of Math.} (2) 73, 295–323, 1961.

\bibitem{swan1962vector}
Richard~G Swan.
\newblock Vector bundles and projective modules.
\newblock {\em Transactions of the American Mathematical Society},
  105(2):264--277, 1962.

\bibitem{Hamiltonianreduction}
Jerrold E Marsden, Gerard Misiolek Juan-Pablo Ortega, Matthew Perlmutter, and Tudor S Ratiu,
\newblock {\em Hamiltonian reduction by stages.} Springer,
\newblock 2007.


\bibitem{mmz} E. Miranda, P. Monnier and N.T.Zung, \emph{ Rigidity of Hamiltonian actions on Poisson manifolds},  Adv. Math. 229 (2012), no. 2, 1136-1179.

\bibitem{miranda2018equivariant}
Eva Miranda and Arnau Planas.
\newblock Equivariant classification of $b^m$-symplectic surfaces
\newblock {\em Regular and Chaotic Dynamics} 355--371, 2018.

\bibitem{GS} Victor Guillemin and Shlomo Sternberg
\newblock A normal form for the moment map.
\newblock {\em Differential geometric methods in mathematical physics} 6:161--175, 1984.

\bibitem{montaldi2009symplectic}
James Montaldi and Juan-Pablo Ortega.
\newblock Symplectic group actions and covering spaces.
\newblock {\em Differential Geometry and its Applications} 27.5:589--604, 2009.
\end{thebibliography}
\end{document}